\documentclass[12pt,reqno]{amsart}

\usepackage{amssymb}
\usepackage{amsmath,amscd}
\usepackage{color}
\usepackage{epsf}
\usepackage{graphicx}

\theoremstyle{plain}
\newtheorem{thm}{Theorem}[section]
\newtheorem{lem}[thm]{Lemma}

\newtheorem{cor}[thm]{Corollary}

\theoremstyle{definition}
\newtheorem{definition}[thm]{Definition}
\newtheorem{example}[thm]{Example}
\newtheorem{rem}[thm]{Remark}

\newcommand{\R}{\mathbb R}
\newcommand{\Z}{\mathbb Z}
\newcommand{\Q}{\mathbb Q}

\newcommand{\nn}{\vskip 0.2cm}
\newcommand{\n}{\vskip 0.1cm}

\begin{document}

\title [\ ] {Localizable invariants of combinatorial manifolds and Euler characteristic}

\author{Li Yu}
\address{Department of Mathematics and IMS, Nanjing University, Nanjing, 210093, P.R.China}
\email{yuli@nju.edu.cn}

 \keywords{Combinatorial manifold,
   localizable invariant, Euler characteristic,
     Pontryagin number, local formula}

\thanks{2010 \textit{Mathematics Subject Classification}.
 57R05, 57R20, 57Q99. \\
   This work is partially supported by
 Natural Science Foundation of China (grant no.11001120)}



\begin{abstract}
   It is shown that if a real value PL-invariant of closed combinatorial
   manifolds admits a local formula that depends
    only on the $\mathbf{f}$-vector of the link of each vertex, then
    the invariant must be a constant times the Euler characteristic.
\end{abstract}

\maketitle

  \section{Introduction}

       For an $n$-dimensional topological manifold $M^n$,
       a combinatorial manifold structure on $M^n$ is a
       triangulation $\xi$ of $M^n$ so that the link of each non-empty
       $i$-simplex in $\xi$ is a PL sphere of dimension
       $n-i-1$.  Here ``PL'' is an abbreviation for \textit{piecewise linear}.
       We call $(M^n,\xi)$ a
        \textit{combinatorial manifold}.
      Note that a simplicial complex which is topologically a
      manifold is not necessarily a combinatorial manifold.
       \nn

   \begin{definition}[Localizable PL-Invariant] \label{Def:Local-Inv}
    A real value invariant $\Psi$ of a closed
     combinatorial $n$-manifold $(M^n,\xi)$ under PL homeomorphisms is called
   \textit{localizable} if there exists a real value
    function $\psi$ on the set of simplicial isomorphism classes of PL $(n-1)$-spheres such that
    \[    \Psi(M^n,\xi) = \sum_{\text{vertex}\, v\,\in\, \xi} \psi(lk(v))     \]
      where $lk(v)$ is the link of a vertex $v$ in the triangulation $\xi$ of $M^n$. We call
       $\psi$ a \textit{local formula} for $\Psi$.
      Let $\mathcal{S}_n$ be the set of simplicial isomorphism
     classes of all PL $(n-1)$-spheres. Then $\psi$ is a
     function $\mathcal{S}_n \rightarrow \R^1$. By our definition, here a localizable invariant $\Psi$
     and its local formula $\psi$
      do not depend on the orientation of manifolds.
    \n

       In addition, if $\psi$ depends only on
       the number of simplices in each dimension in a PL $(n-1)$-sphere,
        we call $\psi$ a \textit{simple local formula} of $\Psi$. In
        this case, we can write
        \begin{equation} \label{Equ:Simple-Local-Formula}
           \Psi(M^n, \xi) = \sum_{\text{vertex}\, v\,\in\, \xi}
       \psi(f_0(lk(v)),\cdots,f_{n-1}(lk(v))),
       \end{equation}
         where $f_k(lk(v))$ is the number of $k$-simplices in
       $lk(v)$. And we call any localizable PL-invariant which admits a simple local formula
        a \textit{simple localizable PL-invariant}.
        \end{definition}
        \n

    \noindent \textbf{Warning:}
      The function $\psi : \mathcal{S}_n \rightarrow \R^1$ itself
       is not an invariant of PL $(n-1)$-spheres under PL
       homeomorphisms.
      \nn
  In the following, we do not distinguish
     a specific PL $(n-1)$-sphere $L$ and the simplicial isomorphism class in $\mathcal{S}_n$ represented by
     $L$ (the meaning should be clear from the context).
     \nn

   \begin{definition}[$\mathbf{f}$-vector]
    For any $L \in \mathcal{S}_n$, let $f_i(L)$ be the number of
    $i$-dimensional simplices in $L$. Then we call
    $$\mathbf{f}(L)=(f_0(L), \cdots, f_{n-1}(L)) \in \Z_+^n$$
    the $\mathbf{f}$\textit{-vector} of $L$.
    In addition, we define $f_{-1}(L) :=1$.
   More generally, for any triangulation $\xi$ of a closed manifold
   $M^n$, let $f_i(M^n, \xi)$ be the number of
    $i$-dimensional simplices in the triangulation and we call
    $$\mathbf{f}(M^n,\xi)=(f_0(M^n,\xi), \cdots, f_{n}(M^n,\xi)) \in \Z_+^{n+1}$$
    the $\mathbf{f}$\textit{-vector} of $(M^n,\xi)$.
    \end{definition}
    \n

      We have the following well-known fact
    on the $\mathbf{f}$-vectors of PL spheres (see~\cite{Brond83}).\nn

   \begin{thm} [Dehn-Sommerville Equations]
    For any $L\in \mathcal{S}_n$,
    \[  f_i(L) = \sum^{n-1}_{j=i} (-1)^{n-1-j} \binom{j+1}{i+1} f_j(L),
   \ -1 \leq i \leq n-1 \]
    In particular, when $i=-1$, the equation gives the Euler
    formula of $L$.
   \end{thm}

    \begin{cor}[see~\cite{Brond83}]
      For any $L\in \mathcal{S}_n$,
    $f_{[\frac{n}{2}]}(L), \cdots , f_{n-1}(L)$ are completely
      determined by $f_0(L), \cdots, f_{[\frac{n}{2}]-1}(L)$.
    \end{cor}
    \n

    \begin{example}
     The Euler characteristic
     $\chi(M^n,\xi) = \overset{n}{\underset{k=0}{\sum}}
      (-1)^k f_k(M^n,\xi)$
      is a simple localizable PL-invariant. A
       simple local formula $\psi_{\chi}$ of $\chi$ is
       given by (see~\cite[Proposition 2.1]{Gaiful05})
      \begin{equation} \label{Equ:Euler}
       \psi_{\chi}(L)
       = 1 + \sum^{n-1}_{k=0} (-1)^{k+1} \frac{f_k(L)}{k+2},\ \, \forall\, L\in
       \mathcal{S}_n.
     \end{equation}
    \end{example}
     \n

  In general, we may ask the following question.\vskip .2cm

  \noindent   \textbf{Question 1:}  Are there any simple localizable
  PL-invariants of combinatorial manifolds which are
   independent from Euler characteristic? \vskip .2cm

    In this paper, we will give a negative answer to Question 1
    by proving the following theorem. \vskip .2cm

    \begin{thm} \label{thm:main}
     Any simple localizable PL-invariant of combinatorial manifolds
      is equal to some constant times Euler characteristic.
    \end{thm}

  \begin{rem}
   It was shown in~\cite{YuLi10} that if a PL-invariant $\Psi$ of closed combinatorial
   manifolds depends only on the $\mathbf{f}$-vector of the manifold,
   then $\Psi$ must depend on the Euler
   characteristic. But our Theorem~\ref{thm:main} here does not follow
   from that result. In fact, even if we assume that two $n$-dimensional combinatorial
   manifolds $(M^n_1,\xi_1)$ and $(M^n_2,\xi_2)$
   have the same Euler characteristic and
   $\mathbf{f}$-vector, we can not guarantee that
   there exists a one-to-one correspondence between the set of links of their vertices so that
   the correspondent links in $(M^n_1,\xi_1)$ and $(M^n_2,\xi_2)$ have the same
   $\mathbf{f}$-vectors.
  \end{rem}

   The paper is organized as following. In section~\ref{Sec2}, we
   recall some basic concepts in combinatorial topology and
   discuss some properties of bistellar moves.
   In section~\ref{Sec3}, we study how a simple local formula of a
   PL-invariant changes under different type of bistellar moves. Then in
   section~\ref{Sec4}, we give a proof of
   Theorem~\ref{thm:main}. In section~\ref{Sec5}, we do some
   easy calculations to verify Theorem~\ref{thm:main} in
   dimension $4$.
   \vskip .2cm

   In addition, since Theorem~\ref{thm:main} is trivial in dimension
   $1$, we always assume the dimension $n$ of a combinatorial manifold is at least $2$
    in the rest of the paper.\\

  \section{Bistellar moves} \label{Sec2}

   We first recall some basic definitions in combinatorial topology
   (see~\cite{BP02} and~\cite{RourSand74} for more detailed exposition).\vskip
   .2cm
   \begin{definition}[Combinatorial Manifold]
     For a simplicial complex $X$, the \textit{star} $St(\sigma)$ of a
     simplex $\sigma$ in $X$ is the subcomplex consisting of all the simplices of
     $X$ that contain $\sigma$ and all their subsimplices. The \textit{link} $lk(\sigma)$ of
     $\sigma$ is the subcomplex consisting of all the simplices $\sigma'$ of
     $X$ with $\sigma'\cap \sigma = \varnothing$ and $\sigma' * \sigma$
     (the \textit{join} of
     $\sigma'$ and $\sigma$) being a simplex in $X$.
     An $n$-dimensional simplicial complex $X$ is called a
    (closed) \textit{combinatorial $n$-manifold} if the link
   of any $i$-simplex in $X$ is an PL $(n-i-1)$-sphere. In
   particular, any PL $n$-sphere is a combinatorial $n$-manifold.
   \end{definition}

     \begin{definition}[Bistellar Move] \label{Def:Bistellar-Move}
        Suppose $W$ is an $n$-dimensional \emph{pure} simplicial complex. Here ``pure'' means that all
        the maximal simplices of $W$ are of dimension $n$. Let
         $\sigma$ be an $(n-i)$-simplex in $W$ ($0\leq i \leq n$) such that its link
        in $W$ is isomorphic to the boundary $\partial\tau$ of an
        $i$-simplex $\tau$ but $\tau$ is not a simplex of $W$. Then the
        operation
        \[   T^{n,i}_{\sigma,\tau} (W) :=
        (\xi \backslash (\sigma * \partial\tau)) \cup (\partial\sigma * \tau)  \]
         is called an $n$-dimensional \textit{bistellar} $i$-\textit{move} on $W$.
         A bistellar $i$-move with $i\geq [\frac{n}{2}]$ is also
         called a \textit{reverse bistellar} $(n-i)$-\textit{move}.
         All the bistellar moves in dimension $2$ and $3$ are shown in
         Figure~\ref{p:2dim_Move} and Figure~\ref{p:3dim_Move}.
          Note that except the bistellar
         $0$-move and reverse $0$-move, all other bistellar moves do not
         change the number of vertices (i.e. $0$-dimensional simplices) in $W$.
  \end{definition}

   \begin{figure}
        \begin{equation*}
        \vcenter{
            \hbox{
                  \mbox{$\includegraphics[width=0.97\textwidth]{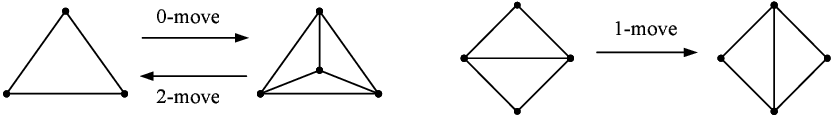}$}
                 }
           }
     \end{equation*}
   \caption{ Bistellar moves in dimension $2$ } \label{p:2dim_Move}
   \end{figure}

    \begin{figure}
        \begin{equation*}
        \vcenter{
            \hbox{
                  \mbox{$\includegraphics[width=\textwidth]{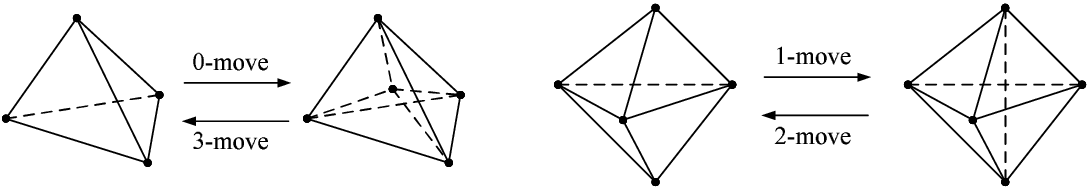}$}
                 }
           }
     \end{equation*}
   \caption{ Bistellar moves in dimension $3$ } \label{p:3dim_Move}
   \end{figure}

    Obviously, any combinatorial manifold $(M^n,\xi)$ is a pure simplicial complex.
      When we do a bistellar move in $(M^n,\xi)$,
      the link of each vertex of $\sigma$ and $\tau$
      involved in the move will be changed simultaneously.
      We have the following observation. \vskip .2cm

   \begin{lem} \label{lem:move}
     Suppose $T^{n,i}_{\sigma,\tau}$ is an $n$-dimensional bistellar $i$-move
      in a combinatorial $n$-manifold $(M^n,\xi)$.
     \begin{itemize}
       \item[(a)] For any $0 < i < n$, $T^{n,i}_{\sigma,\tau}$
          will induce an $(n-1)$-dimensional bistellar $i$-move on the link of
        each vertex of $\sigma$ and induce an $(n-1)$-dimensional bistellar $(i-1)$-move
        on the link of each vertex of $\tau$. \n

        \item[(b)] For $i =0$, $T^{n,0}_{\sigma,\tau}$
        will induce an $(n-1)$-dimensional bistellar $0$-move on the link of
        each vertex of $\sigma$.\n

        \item[(c)] For $i =n$, $T^{n,n}_{\sigma,\tau}$
        will induce an $(n-1)$-dimensional bistellar $(n-1)$-move on the link of
        each vertex of $\tau$.
     \end{itemize}
   \end{lem}
   \begin{proof}
     For each vertex $v_0$ of a $(n-i)$-simplex $\sigma$,
     let $\sigma \backslash \{ v_0 \}$ denote the codimension $1$ face of $\sigma$ that does not contain
     $v_0$. Then the change of $lk(v_0)$ under the bistellar $i$-move $T^{n,i}_{\sigma,\tau}$
     is:
       $$\sigma \backslash \{ v_0 \} * \partial \tau \longrightarrow
     \partial(\sigma \backslash \{ v_0 \}) * \tau, $$
     which by our notation is an $(n-1)$-dimensional
     bistellar $i$-move $T^{n-1,i}_{\sigma \backslash \{ v_0 \},\tau}$.
     Similarly, for any $u_0\in \tau$, the change of $lk(u_0)$ under
     $T^{n,i}_{\sigma,\tau}$ is
      $$\sigma * \partial (\tau \backslash \{ u_0 \}) \longrightarrow
       \partial \sigma * (\tau \backslash \{ u_0 \}),$$
       which is an
       $(n-1)$-dimensional bistellar $(i-1)$-move $T^{n-1,i-1}_{\sigma, \tau \backslash \{ u_0
       \}}$.
   \end{proof} \vskip .2cm

   The relation between bistellar moves and PL-homeomorphisms
    of combinatorial manifolds is shown by the
   following famous theorem of Pachner.\nn

   \begin{thm} [Pachner~\cite{Pa86}] \label{thm:Pachner}
     Two closed combinatorial $n$-manifolds are PL-homeomorphic if
     and only if it is possible to move between their
     triangulations using a sequence of bistellar moves
    and simplicial isomorphisms.
  \end{thm}
   \n

  For any $L\in \mathcal{S}_n$ and
   any $(n-1)$-dimensional bistellar $i$-move on $L$, let $\beta^i\mathbf{f}(L)$
     be the $\mathbf{f}$-vector of $L$ after the move. It is easy to
     see that:
      \begin{equation} \label{Equ:Change}
        \beta^i \mathbf{f}(L) = (f_0(L) + r_{0,i}, \cdots, f_{n-1}(L) + r_{n-1,i}),
        \end{equation}
       $$\mathrm{where}\ \, r_{k,i} = \binom{n-i}{k-i}-\binom{i+1}{n-k}, \ \,
        0 \leq k, i \leq n-1.$$

     $$ \text{Here we define}\ \binom{k}{j} =0 \ \, \text{if $k< j$. It is easy to check
     that:} \qquad\qquad\qquad $$
      \begin{equation} \label{Equ:r1}
          r_{k, n-1-i} = - r_{k,i},\ \, 0 \leq  k, i \leq n-1
      \end{equation}
       \begin{equation} \label{Equ:r2}
       \text{if}\ \, 2i= n-1,\ r_{k,i} =0, \  \, 0 \leq  k\leq n-1
        \end{equation}
        \n

      By~\eqref{Equ:r1}, the reverse bistellar $i$-move on $L$ gives
      \[    \beta^{n-1-i} \mathbf{f}(L) = (f_0(L) - r_{0,i},
      \cdots, f_{n-1}(L) - r_{n-1,i}).          \]

    Suppose $\Psi$ is a localizable PL-invariant of combinatorial
    $n$-manifolds which admits a simple local formula $\psi$.
    Since $\psi$ is simple, by definition we can think of
    $\psi$ as a real value function
    $\mathcal{A}_n \rightarrow \R$ where
    $$\mathcal{A}_n := \{ \mathbf{f}(L) \in \Z_{+}^n \, ; L \in
    \mathcal{S}_n  \}. $$
    So in the following, we also write $\psi(L) = \psi(\mathbf{f}(L))$
     for any $L\in  \mathcal{S}_n$.\nn

    Note that any bistellar move transforms a combinatorial manifold
    to a combinatorial manifold that is PL-homeomorphic to the initial
    one. So $\Psi$ is invariant under all bistellar moves. So for any bistellar
     $i$-move $T^{n,i}_{\sigma,\tau}$ on a combinatorial $n$-manifold,
      the function $\psi$ must satisfy the following equations
      according to Lemma~\ref{lem:move}.\nn

    \begin{itemize}
      \item  When $i\neq 0$ or $n$, we have
      \[
       \sum_{v\in \sigma}\psi(\beta^i\mathbf{f}(lk(v))) +
       \sum_{v'\in \tau}\psi(\beta^{i-1}\mathbf{f}(lk(v'))) =
       \sum_{v\in \sigma}\psi(\mathbf{f}(lk(v))) +
       \sum_{v'\in \tau}\psi(\mathbf{f}(lk(v')))
    \]
    \end{itemize}
   Then we put all the terms corresponding to the vertices of $\sigma$ and
     $\tau$ together and get
  \begin{equation} \label{Equ:i-move}
     \sum_{v\in \sigma} \psi(\beta^i\mathbf{f}(lk(v))) -
     \psi(\mathbf{f}(lk(v))) +
       \sum_{v'\in \tau} \psi(\beta^{i-1}\mathbf{f}(lk(v')))
       - \psi(\mathbf{f}(lk(v'))) = 0
       \end{equation}

    \begin{itemize}
      \item    When $i=0$, we have
      \begin{equation} \label{Equ:0-move}
        \sum_{v\in \sigma} \psi(\beta^0\mathbf{f}(lk(v))) -
     \psi(\mathbf{f}(lk(v)))   + \psi(\mathbf{f}(\partial\Delta^n)) = 0
       \end{equation}
       \item   When $i=n$, we have
     \begin{equation} \label{Equ:n-move}
        -\psi(\mathbf{f}(\partial\Delta^n)) + \sum_{v'\in \tau} \psi(\beta^{n-1}\mathbf{f}(lk(v'))) -
     \psi(\mathbf{f}(lk(v'))) = 0
       \end{equation}
    \end{itemize}
    \noindent
    $$ \text{ where $\mathbf{f}(\partial\Delta^n)=
      \left(\binom{n+1}{1},  \binom{n+1}{2}, \cdots,  \binom{n+1}{n} \right)$
   is the $\textbf{f}$-vector}$$
   of the boundary of the $n$-dimensional simplex $\Delta^n$.
    \nn

     \begin{rem} \label{Rem:Universal_0_1}
     For a given triangulation of a closed manifold $M^n$ and
      an arbitrarily chosen vertex $v$ in it,
     $v$ may not be directly involved in any bistellar $i$-move when $ 2 \leq i \leq n$.
     This is because there may not be any $(n-i)$-simplex $\sigma$ in $St(v)$ whose link
     satisfies the condition of a bistellar $i$-move (see Definition~\ref{Def:Bistellar-Move}).\\
     \end{rem}

 \section{How the value of a simple local formula changes under bistellar moves} \label{Sec3}

      In this section, we first introduce a special type of
       PL $n$-disk in each dimension $n\geq 2$, and then use it to
       show how the value of a simple local formula of a localizable
       invariant changes under bistellar moves. \vskip .2cm

     \begin{lem} \label{Lem:Aux-Cell}
      For each $n\geq 2$,
      there exists a PL $n$-disk $K^n$ and a vertex $v_0 \in \partial K^n$
       such that:
        \begin{enumerate}
            \item[(a)] $\partial K^n$ is isomorphic to the boundary of an
                   $n$-simplex.
            \item[(b)] for any $0\leq i \leq n-1$, there exists a bistellar
            $i$-move $T^{n,i}_{\sigma,\tau}$ in the interior of $K^n$ with
            $v_0\in \sigma \subset K^n$
           so that $T^{n,i}_{\sigma,\tau}$
            does not cause any changes to the star of any vertex on $\partial
            K^n$ except $v_0$.
        \end{enumerate}
    \end{lem}
    \begin{proof}
           For each $0\leq i \leq n-1$, let $\Delta^i$ be a simplex of
        dimension $i$.
         Let $J_i = \Delta^{n-i} * \partial\Delta^i $ and choose
         a vertex $b_0^i$ of $\Delta^{n-i}$ in $J_i$.
         Let $J$ be the one-point union of
         $J_0, \cdots, J_{n-1}$ got by gluing each $b_0^i$ to a point
         $b_0$. On the other hand, let $\widetilde{\Delta}^n_1, \widetilde{\Delta}^n_2$ be
         two $n$-simplices such that $\widetilde{\Delta}^n_2\subset \widetilde{\Delta}^n_1$
         and $\widetilde{\Delta}^n_2 \cap \partial \widetilde{\Delta}^n_1$ is a vertex $v_0$ of both.
         Next, we glue $b_0$ to $v_0$ and put $J$ inside $\widetilde{\Delta}^n_2$
          such that $J \cap \partial \widetilde{\Delta}^n_2 = v_0$.
          By introducing some new simplices
           in $\widetilde{\Delta}^n_1 - J$,
          we can subdivide $\widetilde{\Delta}^n_1$ into a PL $n$-disk such that
          the triangulation of $\partial \widetilde{\Delta}^n_1$
           is not changed. We denote this PL $n$-disk by $K^n$.
           So $\partial K^n$ is isomorphic to the boundary of an
          $n$-simplex (see Figure~\ref{p:Magic_Cell} for a
          construction of $K^2$).\n

           The canonical bistellar
          $i$-move $T^{n,i}_{\sigma,\tau}$ in $K^n$ is just replacing $J_i$
           by $\partial \Delta^{n-i} * \Delta^i$. It is easy to see
           that $v_0\in \sigma$ and $T^{n,i}_{\sigma,\tau}$ will not change the star of
          any vertex on $\partial K^n$ except $v_0$.
          So such a $K^n$ satisfies all our requirements.
    \end{proof} \vskip .2cm

   \begin{figure}
        \begin{equation*}
        \vcenter{
            \hbox{
                  \mbox{$\includegraphics[width=0.87\textwidth]{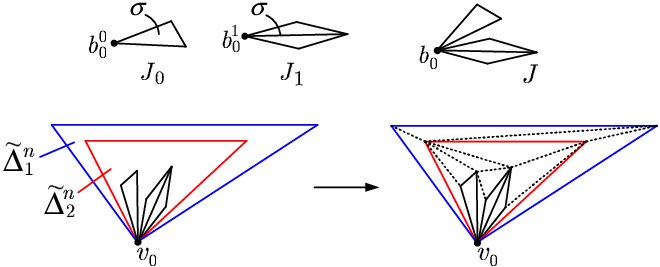}$}
                 }
           }
      \end{equation*}
     \caption{ An auxiliary cell in dimension $2$ } \label{p:Magic_Cell}
   \end{figure}

        Obviously, the kind of $K^n$ that satisfies the conditions in Lemma~\ref{Lem:Aux-Cell}
        is not unique. So in the rest of this paper,
        we fix one such $K^n$ for each dimension $n$.
        We call $K^n$
         the \textit{auxiliary $n$-cell} and $v_0$ the
        \textit{base point} of $K^n$. In addition,
         the canonical $n$-dimensional bistellar $i$-move
        associated to $J_i$ in $K^n$ is denoted by $T^{n,i}(K^n)$.\nn

         Let $a_{n,i}$ be the number of $i$-simplices in the link of $v_0$
        in $K^n$, and let
         $\mathbf{a_n} := (a_{n,0}, \cdots , a_{n,n-1}) - (\binom{n}{1},\cdots, \binom{n}{n}) \in \Z^n_+$. Then define
          $$\mathcal{A}'_n := \{  \mathbf{f}(L) +  \mathbf{a_n} \in
    \Z_{+}^n \, ; L \in \mathcal{S}_n \}.$$
    \nn

    \begin{lem}
       $\mathcal{A}'_n \subset \mathcal{A}_n \subset \Z^n_{+}$.
    \end{lem}
    \begin{proof}
        For any $L \in \mathcal{S}_n$, choose a vertex $u_0$ and
        let $U = u_0 * L$ be a PL $n$-ball whose boundary is $L$ (see Figure~\ref{p:Induced-Move}).
        So the link of $u_0$ in $U$ is isomorphic to $L$.
        Next, we choose an arbitrary $n$-simplex
        in $U$ and subdivide it into an auxiliary $n$-cell $K^n$ so that
        $u_0$ is the base point.
        Then the link of $u_0$ in $U$ becomes a new PL
        $(n-1)$-sphere $L'$
         whose $\mathbf{f}$-vector differs from $\mathbf{f}(L)$ by $\mathbf{a_n}$.
         Hence $\mathbf{f}(L) +  \mathbf{a_n} = \mathbf{f}(L')\in \mathcal{A}_n$.
    \end{proof}

    \begin{figure}
        \begin{equation*}
        \vcenter{
            \hbox{
                  \mbox{$\includegraphics[width=0.49\textwidth]{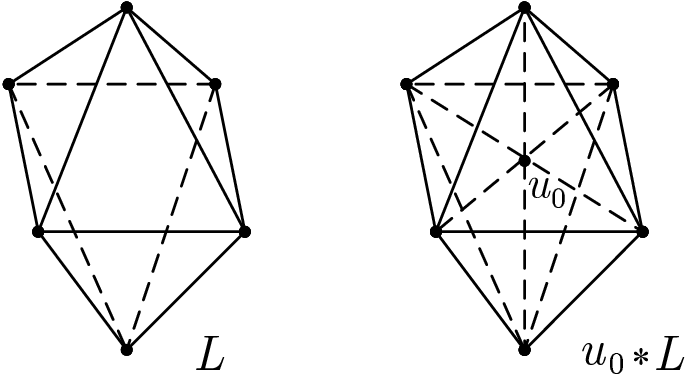}$}
                 }
           }
      \end{equation*}
     \caption{ } \label{p:Induced-Move}
   \end{figure}

    \n

   \begin{lem}\label{Lem:const_1}
   If $\psi$ is simple local formula of a localizable
    invariant of combinatorial $n$-manifolds, then
    for any $0\leq i \leq n-1$ and any $\mathbf{f'} \in \mathcal{A}'_n$,
    $\psi(\beta^i\mathbf{f'}) - \psi(\mathbf{f'})$ is independent on
    $ \mathbf{f'}$.
    \end{lem}
   \begin{proof}
        For an element $\mathbf{f}' \in \mathcal{A}'_n$,
        let $L \in \mathcal{S}_n$ with
        $\mathbf{f}(L) + \mathbf{a_n} = \mathbf{f'}$.
        Suppose $v$ is a vertex in a combinatorial manifold $(M^n,\xi)$
        such that $lk(v) \cong L$. We choose an $n$-simplex
        in $St(v)$ and subdivide it into an auxiliary $n$-cell $K^n$ so that
        $v$ is the base point. Then after the subdivision, we have $\mathbf{f}(lk(v))=\mathbf{f}(L) +
        \mathbf{a_n}=\mathbf{f'}$.\n

        For any $0\leq i \leq n-1$,
        We do the canonical bistellar $i$-move $T^{n,i}(K^n)$
        in the auxiliary cell $K^n$. Let $u^i_1,\cdots, u^i_{n+1}$ be all the
         vertices involved in $T^{n,i}(K^n)$ other than $v$.
         By the construction of $K^n$ (see Lemma~\ref{Lem:Aux-Cell}),
         for each $ 1\leq j \leq n+1$, the
        star of $u^i_j$ is contained in the interior of $K^n$. So the change of $St(u^i_j)$
        under $T^{n,i}(K^n)$ is canonically determined by $K^n$.
         So if we write down the Equation~\eqref{Equ:i-move}
         or~\eqref{Equ:0-move} for $T^{n,i}(K^n)$,
         all the terms in the equation are
         canonically determined by $K^n$ except
         $\psi(\beta^i\mathbf{f}(lk(v))) - \psi(\mathbf{f}(lk(v)))$.
         So $\psi(\beta^i\mathbf{f'}) - \psi(\mathbf{f'})$ is determined only by $K^n$,
         but independent on the value of $\mathbf{f'}$.
   \end{proof} \vskip .2cm

    In the rest of this section, we fix a
    simple localizable PL-invariant $\Psi$ of combinatorial
    $n$-manifolds and let $\psi$ be a simple local formula of $\Psi$.
    By Lemma~\ref{Lem:const_1},
    for any $ \mathbf{f'} \in \mathcal{A}'_n$, let
     $$\psi(\beta^i\mathbf{f'}) - \psi(\mathbf{f'}) := H^n_i(\psi) \,\in \R^1,\ 0\leq i \leq n-1$$
     where $H^n_i(\psi)$ is independent on $\mathbf{f}'$.
     In addition, we define
     $$H^n_{-1}(\psi) := \psi(\mathbf{f}(\partial\Delta^n)).\qquad\qquad $$
        The following lemma tells us some relations between different
        $H^n_i(\psi)$'s.   \vskip .2cm

   \begin{lem} \label{Equ:d_i}
     For any $0\leq i \leq n-1$, we have:
     \begin{enumerate}
       \item[(a)] each $H^n_i(\psi)$ is a rational multiple of
       $\psi(\mathbf{f}(\partial\Delta^n))$.\n
       \item[(b)] $(n-i+1)\cdot H^n_i(\psi) + (i+1)\cdot H^n_{i-1}(\psi) =0 $.\n
       \item[(c)] $H^n_i(\psi) = -H^n_{n-i-1}(\psi)$.\n
       \item[(d)] when $n$ is odd, $H^n_i(\psi)=0$ for any $-1\leq i \leq n-1$.
     \end{enumerate}
   \end{lem}
   \begin{proof}
       Suppose we have a bistellar $i$-move $T^{n,i}_{\sigma,\tau}$ in a
       combinatorial $n$-manifold $(M^n,\xi)$.
       For each vertex $v$ of $\sigma$,
        we choose an $n$-simplex
        $\Delta^n_v$ in $St(v)$ that does not contain $\sigma$,
        and then subdivide $\Delta^n_v$ into an auxiliary cell $K^n$ so
        that $v$ is the base point
        (sharing the same auxiliary cell between different stars are allowed).
       Then the vector $\mathbf{f}(lk(v))$ becomes an element of $\mathcal{A}_n'$
       after the subdivision. Similarly, for each vertex $v'$ of $\tau$,
       we choose an $n$-simplex $\Delta^n_{v'}$ in $St(v')$ that does not contain $\sigma$
       and do the same subdivision (see Figure~\ref{p:Subdivision}).
         \begin{figure}
        \begin{equation*}
        \vcenter{
            \hbox{
                  \mbox{$\includegraphics[width=0.92\textwidth]{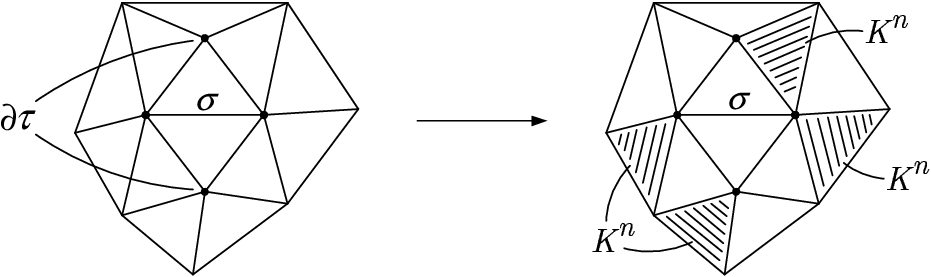}$}
                 }
           }
      \end{equation*}
     \caption{ } \label{p:Subdivision}
   \end{figure}
    \n

        Let $\xi'$ be the new triangulation of $M^n$ after these
        subdivisions in $\xi$. Notice that the
        bistellar $i$-move $T^{n,i}_{\sigma,\tau}$ can still be carried out in
        $(M^n,\xi')$.
         Then by Lemma~\ref{Lem:const_1}, we obtain (b) from
       Equation~\eqref{Equ:i-move} and
       ~\eqref{Equ:0-move} for the bistellar $i$-move $T^{n,i}_{\sigma,\tau}$ in $(M^n,\xi')$.
      Moreover,
       we can easily write each $H^n_i(\psi)$ ($0\leq i \leq n-1$) as a rational multiple of
      $H^n_{-1}(\psi)=\psi(\mathbf{f}(\partial\Delta^n))$ by using (b) recursively.
      In particular, we get
        $$H^n_i(\psi) = (-1)^{n-2i-1} H^n_{n-i-1}(\psi).$$
       Then (c) is true when $n$ is even.
      When $n$ is odd,
      we notice that
      $$\beta^{\frac{n-1}{2}} \mathbf{f} = \mathbf{f},\ \, \forall\,\mathbf{f} \in
      \mathcal{A}_n.$$
      So
       $H^n_{\frac{n-1}{2}}(\psi) =0$, which implies that $\psi(\mathbf{f}(\partial\Delta^n))=0$
       hence $H^n_i (\psi)=0$ for all $0\leq i\leq n-1$. So
       (c) still holds when $n$ is odd. The lemma is proved.
   \end{proof}

   \n

   \begin{lem} \label{Lem:const}
      Suppose $\Psi$ is a localizable
   PL-invariant of combinatorial $n$-manifolds which admits a simple local formula $\psi$.
    Then for any PL $(n-1)$-sphere $L\in \mathcal{S}_n$ and any $(n-1)$-dimensional bistellar
    $i$-move $T^{n-1,i}_{\sigma,\tau}$ on $L$, we have
    $$\psi(\mathbf{f}(T^{n-1,i}_{\sigma,\tau}(L))) - \psi(\mathbf{f}(L)) = H^n_i(\psi),
    \    0\leq  i \leq n-1.$$
   \end{lem}
 \begin{proof}
  Let $S^0 = \{ u_0, u_1 \}$ be a pair of disjoint points. Then $S^0 * L$ is a PL
  $n$-sphere called the \emph{suspension} of $L$. Obviously,
  the links of $u_0$ and $u_1$ in $S^0 * L$ are both $L$.
  In addition, let $\widetilde{\sigma} = u_0 * \sigma$. Then
  the $n$-dimensional bistellar $i$-move $T^{n,i}_{\widetilde{\sigma},\tau}$ in
  $S^0 * L$ will induce the bistellar move $T^{n-1,i}_{\sigma,\tau}$ on $L$ (see Lemma~\ref{lem:move} (a)).\n

    Next, we subdivide each $n$-simplex in $St(u_1)$
    into the auxiliary cell $K^n$ and
     let $\widetilde{L}$ denote the new PL $n$-sphere obtained after these subdivisions
    on $S^0 * L$.
     Note that $St(u_0)$ stays unchanged under these subdivisions (because of the construction
    of $K^n$). Now for any vertex $v$ of $\sigma$ or $\tau$ and
    the link $lk(v)$ of $v$ in $\widetilde{L}$,
    we have $\mathbf{f}(lk(v)) \in \mathcal{A}_n'$.
   \n

  We can still do the bistellar $i$-move $T^{n,i}_{\widetilde{\sigma},\tau}$ in $\widetilde{L}$.
    Then from~\eqref{Equ:i-move},~\eqref{Equ:0-move} and Lemma~\ref{Lem:const_1}, we get:
    \[
   \psi(\beta^i\mathbf{f}(lk(u_0))) - \psi(\mathbf{f}(lk(u_0)))
      + (n-i) H^n_i(\psi) + (i+1)H^n_{i-1}(\psi) = 0.
       \]
    Since $lk(u_0) =L$ and $\beta^i\mathbf{f}(lk(u_0)) =
       \mathbf{f}(T^{n-1,i}_{\sigma,\tau}(L))$, so we have
        \[
  \psi(\mathbf{f}(T^{n-1,i}_{\sigma,\tau}(L))) - \psi(\mathbf{f}(L))
      + (n-i) H^n_i(\psi) + (i+1)H^n_{i-1}(\psi) = 0.
       \]
    Moreover, Lemma~\ref{Equ:d_i} (b) tells us that
    $$ (n-i) H^n_i(\psi) + (i+1)H^n_{i-1}(\psi) = -H^n_i (\psi).$$
     So $\psi(\mathbf{f}(T^{n-1,i}_{\sigma,\tau}(L))) - \psi(\mathbf{f}(L))
      = H^n_i (\psi)$.
 \end{proof}\n

        Lemma~\ref{Lem:const} implies that
        the change of a simple
        local formula $\psi$ on a PL sphere caused by a bistellar $i$-move does not depend on
        where the bistellar $i$-move takes place in the PL sphere.
        This is a very strong restriction on a PL-invariant that could
         admit a simple local formula. This somehow explains why the
          conclusion in Theorem~\ref{thm:main} could be true.  \\

  \section{Proof of Theorem~\ref{thm:main}} \label{Sec4}

     Suppose $L$ is an arbitrary PL $(n-1)$-sphere.
      By Theorem~\ref{thm:Pachner}, we can
      use a finite sequence of $(n-1)$-dimensional
      bistellar moves to turn $\partial \Delta^n$
      into $L$. For each $0\leq i \leq n-1$, suppose there are
      $m_i(L)$ bistellar $i$-moves in the sequence.
      Then by~\eqref{Equ:Change}, $m_0(L),\cdots, m_{n-1}(L)$ satisfy
       \begin{equation} \label{Equ:steps}
          \sum^{n-1}_{i=0} m_i(L) \cdot r_{k,i} = f_k(L) - f_k(\partial\Delta^n),\ \,
           0 \leq k \leq n-1.
     \end{equation}
     By Equation~\eqref{Equ:r1} and ~\eqref{Equ:r2}, the linear system~\eqref{Equ:steps} is equivalent to:
        \begin{equation} \label{Equ:System1}
             \overset{[\frac{n}{2}]-1}{\underset{i=0}{\sum}} \left(
           m_i(L) - m_{n-1-i}(L) \right) \cdot r_{k,i} = f_k(L) - f_k(\partial\Delta^n),
           \ \, 0 \leq k \leq n-1
        \end{equation}
        \n

        Since there exist  many different ways to turn
        $\partial \Delta^n$ to $L$ via bistellar moves, so $m_i(L)$ is not
        canonically determined by $L$. But the following lemma
        (which reformulates a result in~\cite{Pa86}) tells us that the
        difference $m_i(L) - m_{n-1-i}(L)$ is actually uniquely
       determined by $L$.\nn

       \begin{lem}\label{Lem:Unique}
        For any $L \in \mathcal{S}_n$ and $0\leq i \leq [\frac{n}{2}]-1$,
         $m_i(L) - m_{n-1-i}(L)$ in~\eqref{Equ:System1} is
         uniquely determined by $f_0(L),\cdots f_{n-1}(L)$.
       \end{lem}
       \begin{proof}
        Let us only consider the first
         $[\frac{n}{2}]$ equations in the system~\eqref{Equ:System1}:
             \begin{equation} \label{Equ:System2}
                \overset{[\frac{n}{2}]-1}{\underset{i=0}{\sum}} \left(
           m_i(L) - m_{n-1-i}(L) \right) \cdot r_{k,i} = f_k(L) - f_k(\partial \Delta^n),
           \ \, 0 \leq k \leq \left[\frac{n}{2} \right]-1.
            \end{equation}

          Notice when $0\leq i\leq [\frac{n}{2}]-1$, $0\leq k\leq
     [\frac{n}{2}]-1$, $r_{k,i}=\binom{n-i}{k-i}$. So
        \begin{itemize}
         \item if $k< i$, $r_{k,i} =0$.\n
         \item if $k=i$, $r_{i,i} = 1$.
        \end{itemize}
       So the square integral matrix $(r_{k,i})_{0\leq k,i\leq [\frac{n}{2}]-1 }$
       is invertible over $\Z$. So the linear system~\eqref{Equ:System2}
       has a unique solution.
       \end{proof}

        \begin{rem}
        When $n=2s+1$ is odd, from Equation~\eqref{Equ:r2},
        $r_{k,s} =0$ for any $0\leq k \leq n-1$. So
         in~\eqref{Equ:System1},
         the term $m_s\cdot r_{k,s}$ is omitted.
       \end{rem} \vskip .2cm

      By the proof of Lemma~\ref{Lem:Unique},
      for any $L\in \mathcal{S}_n$ and $0\leq i\leq [\frac{n}{2}]-1$,
      \begin{equation} \label{Equ:Solution}
       \quad m_i(L) - m_{n-1-i}(L) =
      \sum^{[\frac{n}{2}]-1}_{k=0} c_{ik} \left( f_k(L) -  f_k(\partial\Delta^n) \right) ,
       \end{equation}
       where $\{ c_{ik} \in \Z \}_{0\leq i,k\leq [\frac{n}{2}]-1}$ are some universal constants.
      \\

 \noindent \textbf{Proof of Theorem~\ref{thm:main}:}
 Suppose $\Psi$ is a simple localizable PL-invariant of
       combinatorial $n$-manifolds and $\psi$ is a simple local formula of $\Psi$.
       By our discussion above, for any $L \in
         \mathcal{S}_n$, we have
      \begin{align} \label{Equ:psi_formula}
       \psi(\mathbf{f}(L)) & = \psi(\mathbf{f}(\partial\Delta^n)) +
       \sum^{n-1}_{i=0} m_i(L)\cdot H^n_i(\psi)\quad \text{(by Lemma~\ref{Lem:const})}
        \notag \\
        & = \psi(\mathbf{f}(\partial\Delta^n)) + \sum^{[\frac{n}{2}]-1}_{i=0}
         (m_i(L)-m_{n-1-i}(L)) \cdot H^n_i(\psi)  \\
        & \overset{\eqref{Equ:Solution}}{=}  \psi(\mathbf{f}(\partial\Delta^n)) +
         \sum^{[\frac{n}{2}]-1}_{i=0} \sum^{[\frac{n}{2}]-1}_{k=0}
         H^n_i(\psi)\cdot c_{ik}  \left( f_k(L) -  f_k(\partial\Delta^n)
         \right).
         \notag \ \
        \end{align}
      So $ \psi(\mathbf{f}(L))$ is a linear function of $f_0(L),\cdots,
      f_{n-1}(L)$. Moreover, since
      each $H^n_i(\psi)$ is a rational multiple of $\psi(\mathbf{f}(\partial\Delta^n))$, so
       we can write
      \[   \psi(\mathbf{f}(L)) = \psi(\mathbf{f}(\partial\Delta^n))\cdot \sum^{[\frac{n}{2}]-1}_{k=-1} q_k\cdot f_k(L),
         \;\  q_k \in \Q \quad (\text{recall}\ f_{-1}(L):=1). \]

       Then for a combinatorial $n$-manifold $(M^n,\xi)$, we have
       \begin{align} \label{equ:Psi}
            \Psi(M^n,\xi) & = \sum_{ v\,\in\, \xi} \psi(\mathbf{f}(lk(v))) \notag \\
            &= \psi(\mathbf{f}(\partial\Delta^n))\cdot \sum_{v\,\in\, \xi}  \sum^{[\frac{n}{2}]-1}_{k=-1} q_k \cdot
            f_k(lk(v))  \notag \\
            &= \psi(\mathbf{f}(\partial\Delta^n))\cdot \sum^{[\frac{n}{2}]-1}_{k=-1} q_k \cdot \left(\sum_{v\,\in\, \xi}
            f_k(lk(v)) \right).
       \end{align}

       Let $f_k(M^n,\xi)$ be the number of $k$-simplices in $\xi$. Then obviously
        \[   f_k(M^n,\xi)
        = \frac{1}{k+1} \sum_{v\,\in\, \xi} f_{k-1}(lk(v)), \ 0\leq k \leq n \]
        \[
             \Longrightarrow\ \,  \Psi(M^n,\xi) =
            \psi(\mathbf{f}(\partial\Delta^n))\cdot
        \sum^{[\frac{n}{2}]-1}_{k=-1} q_k (k+2)\cdot
        f_{k+1}(M^n,\xi).
       \]
         So $\Psi(M^n,\xi)$ is a linear function of $f_0(M^n,\xi), \cdots,
         f_{n}(M^n,\xi)$. Then by the Theorem~\ref{thm:Justin} below,
         $\Psi$ must be a constant times the Euler characteristic.
    \hfill $\square$
      \nn

      \begin{thm} [Roberts~\cite{Justin02}] \label{thm:Justin}
   Any linear combination
    of the numbers of simplices which is an invariant of closed combinatorial manifolds
    under PL homeomorphism must be proportional to the Euler
    characteristic.\\
    \end{thm}

   \section{Verification of Theorem~\ref{thm:main} in dimension $4$}
   \label{Sec5}

      When $n=4$, by the Dehn-Sommerville equations,
     the $\mathbf{f}$-vector of a PL $3$-sphere $L$ depends only on
      the number of vertices and edges in $L$.
      So if $\psi$ is a simple local formula of a PL-invariant $\Psi$,
      we can write
      \[ \psi(L) = \psi(\mathbf{f}(L))=
      \psi(f_0(L), f_1(L)),\ \, \forall\, L \in \mathcal{S}_4 . \]
     In this case, the linear system~\eqref{Equ:System2} reads:
    \begin{align*}
       m_0(L) - m_3(L)  &= f_0(L) -5, \\
       4(m_0(L)-m_3(L)) + (m_1(L)-m_2(L)) &= f_1(L) -10. \qquad \ \
    \end{align*}
   So $m_0(L) - m_3(L) = f_0(L) -5$, $m_1(L)-m_2(L) = f_1(L) - 4f_0(L) +10$.
   In addition, by Lemma~\ref{Equ:d_i}, we have
   $$ H^4_0(\psi)= -\frac{1}{5} \psi(\mathbf{f}(\partial\Delta^4)) ;  \quad
    H^4_1(\psi) = \frac{1}{10} \psi(\mathbf{f}(\partial\Delta^4)). $$
    Then by
   Equation~\eqref{Equ:psi_formula}, we have
        \begin{align*}
         \psi(L) &= \psi(\mathbf{f}(\partial\Delta^4))
         \cdot \left( 1 -\frac{1}{5}(f_0(L) -5) +
          \frac{1}{10} (f_1(L) - 4f_0(L) +10) \right)  \\
           &=  3 \psi(\mathbf{f}(\partial\Delta^4))\cdot \left( 1 -\frac{1}{5} f_0(L)
          + \frac{1}{30} f_1(L)    \right).
         \end{align*}
   On the other hand, a local formula $\psi_{\chi}$
    for the Euler characteristic $\chi$ of
    a $4$-dimensional
     combinatorial manifold is (see~\eqref{Equ:Euler}):
      \[     \psi_{\chi}(L) = 1 - \frac{f_0(L)}{2} + \frac{f_1(L)}{3} -
      \frac{f_2(L)}{4} + \frac{f_3(L)}{5}, \ \, \forall\, L \in \mathcal{S}_4.
                              \]
     The Dehn-Sommerville equations for the PL $3$-spheres imply that
         \[    f_2(L)=2f_3(L), \quad f_3(L)=f_1(L) -
         f_0(L).
           \]
     \[    \Longrightarrow \ \,
         \psi_{\chi}(L) = 1 - \frac{f_0(L)}{5} +
         \frac{f_1(L)}{30}. \qquad\qquad
     \]
     So we have $\psi (L) = 3\psi(\mathbf{f}(\partial\Delta^4)) \cdot \psi_{\chi} (L)$.
     Then for any $4$-dimensional combinatorial manifold
     $(M^4,\xi)$, we have
     $$\Psi(M^4,\xi)  = 3 \psi(\mathbf{f}(\partial\Delta^4)) \cdot \chi(M^4,\xi).$$
   This verifies Theorem~\ref{thm:main} in dimension $4$.

    \ \\

\end{document}